\documentclass[12pt]{article}
\usepackage[margin=2cm]{geometry}
\usepackage{amsmath,amssymb,amstext,amsthm,bbm,xcolor,url}
\usepackage{hyperref}

\theoremstyle{plain}    
 \newtheorem{cor}{Corollary}[section] 
 \newtheorem{cl}[cor]{Claim} 
 \theoremstyle{plain}    
 \newtheorem{thm}{Theorem} 
 \theoremstyle{plain}    
 \newtheorem{lemma}[cor]{Lemma} 
 \theoremstyle{plain}    
 \newtheorem{prop}[cor]{Proposition} 
 \theoremstyle{definition}
 
 \theoremstyle{remark}
 \newtheorem*{rem*}{Remark}

\def\PSL{\operatorname{PSL}}
\def\SL{\operatorname{SL}}
\def\PGL{\operatorname{PGL}}

\def\u[#1]{\widetilde{C_2^{#1}}}

\def\stackbelow#1#2{\underset{\displaystyle\overset{\displaystyle\shortparallel}{#2}}{#1}}

\begin{document}

\author{Lior Bary-Soroker\textsuperscript{1}, Daniele Garzoni\textsuperscript{2}, Sasha Sodin\textsuperscript{3}}
\title{Irreducibility of the characteristic polynomials of random tridiagonal matrices}
\maketitle

\begin{abstract} Conditionally on the Riemann hypothesis for certain Dedekind zeta functions, we show that the characteristic polynomial of a class of random tridiagonal matrices of large dimension is irreducible, with probability exponentially close to one; moreover, its Galois group over the rational numbers is either the symmetric or the alternating group. This is the counterpart of the results of Breuillard--Varj\'u (for polynomials with independent coefficients), and with those of Eberhard and Ferber--Jain--Sah--Sawhney (for full random matrices). We also analyse a related class of random tridiagonal matrices for which the Galois group is much smaller.
\end{abstract}

\footnotetext[1]{Raymond and Beverly Sackler School of Mathematical Sciences, Tel Aviv University, Tel Aviv 6997801, Israel.
\mbox{E-mail}: barylior@tauex.tau.ac.il}
\footnotetext[2]{Department of Mathematics, University of Southern California, Los Angeles, CA 90089-2532, USA. 
\mbox{E-mail}: garzoni@usc.edu}
\footnotetext[3]{Einstein Institute of Mathematics, Edmond J. Safra Campus, Givat Ram, The Hebrew University of Jerusalem, Jerusalem, 9190401, Israel \& School of Mathematical Sciences, Queen Mary University of London, Mile End Road, London E1 4NS, UK.
\mbox{E-mail}: alexander.sodin@mail.huji.ac.il}

\section{Introduction}

\noindent Let $V_1, V_2, V_3, \cdots$ be independent, identically distributed, non-constant integer-valued random variables with
\begin{equation}\label{eq:tail}
\mathbb P\{ |V_1| \geq k \} \leq B \log^{-b} k~, \quad k \geq 3~,
\end{equation}
for some $B, b > 0$. 
Let $P_n(\lambda) = \det(\lambda \mathbbm{1} - H_n)$ be the characteristic polynomial of 
\begin{equation}\label{eq:defHn}
H_n = \left( \begin{array}{cccccc} 
V_1 & 1 & 0 & 0 & 0 & \cdots \\
1 & V_2 & 1 & 0 & 0 & \cdots \\
0 & 1 & V_3 & 1 & 0 & \cdots \\
\vdots&     &   \ddots & \ddots & \ddots &   \\ 
0 &&\cdots&1 &V_{n-1} & 1 \\
0 && \cdots &0 & 1 & V_n \\ 
\end{array} \right)~. \end{equation}

The main result we present is conditional on the extended Riemann Hypothesis for polynomials $Q \in \mathbb Z[\lambda]$:
\begin{equation}\label{eq:rh}
\begin{split}
&\text{all the non-trivial zeros of the Dedekind zeta functions of the number fields $\mathbb Q(a)$,}\\
&\text{where $a$ is a root of $Q$, lie on the critical line.}
\end{split}
\end{equation}

\begin{thm}\label{thm:1}
Assume  the extended Riemann Hypothesis (\ref{eq:rh}) for all the characteristic polynomials of matrices of the form \eqref{eq:defHn} with $V_j \in \mathbb Z$. There exist $C, c > 0$ such that the following holds with probability $\geq 1 - C \exp(-c n)$: $P_n$ is irreducible over $\mathbb Q$, and, moreover, the Galois group $\operatorname{Gal}(P_n / \mathbb Q)$ is either the symmetric group $S_n$ or the alternating group $A_n$. 
\end{thm}

Here $\operatorname{Gal}(P_n/\mathbb Q)$ is the Galois group of the splitting field of $P_n$ over $\mathbb Q$
 endowed with the action on its roots. The exponent $c$ and the prefactor $C$ in the conclusion of the theorem  depend only on  the constants $B, b$ in (\ref{eq:tail}) and on the second largest probability mass
\begin{equation}\label{eq:atombd}\max \left\{ \min(\mathbb P(V_1 = v), \mathbb P(V_1 =   v')) \, : \, v, v' \in \mathbb Z~, \, v \neq v'\right\}~.\end{equation}

The estimate on the probability is sharp, up to the values of the constants, since with probability $(\mathbb P\{V_1 = v\}) ^{n}$ all $V_j \equiv v$ and $P_n$ is reducible for $n \geq 2$ (in fact, $P_n(x) = U_n((x-v)/2)$, where $U_n(\cos \theta) = \frac{\sin((n+1)\theta)}{\sin \theta}$ is the $n$-th Chebyshev polynomial of the second kind). Similarly to many of the models we discuss below, it would be interesting to rule out the possibility that $\operatorname{Gal}(P_n/\mathbb Q)= A_n$. We leave this as an interesting open problem.

\subsection{Motivation}
The motivation for this result comes from the belief that ``generic'' polynomials of high degree with integer coefficients are irreducible over $\mathbb Q$. In other words, if $\mathcal P_n$ is a sufficiently rich family of polynomials of degree $n$ which have  no explicit algebraic reason to be reducible, we expect that
\begin{equation} \label{eq:mainconj} \# \{ P \in \mathcal P_n~, \,\, \text{$P$ is reducible} \} / \# \mathcal P_n \to 0~, \quad n \to \infty~. \end{equation}
A stronger belief is 
\begin{equation} \label{eq:mainconj2} \# \{ P \in \mathcal P_n~, \,\, \operatorname{Gal}(P/\mathbb Q) \neq S_n \} / \# \mathcal P_n \to 0~, \quad n \to \infty~. \end{equation}

These vaguely stated beliefs have been extensively examined for polynomials with independent coefficients, i.e.\ 
\[ \mathcal P_n^{\text{ind}, A} = \left\{ P(\lambda) = \lambda^n + \sum_{j=0}^{n-1} a_j \lambda^{n-j}~, \quad a_j \in A \right\}~, \]
where $A \subset \mathbb Z$ is a set of integers of cardinality $2 \leq \#A < \infty$.

If $0\in A$, then with positive probability, $P(0)=0$, and $P$ is reducible. Thus assume that $0 \not\in A$ (an alternative is to consider $\mathcal P^{\text{ind},A} \cap \{P(0) \neq 0\}$). The first author and Kozma showed \cite{BSK} that if 
\begin{equation}\label{eq:assumbsk}\begin{split} 
&\text{there exist four primes $p$ such that each residue class modulo $p$} \\ 
&\text{contains the same number of elements of $A$}\end{split}\end{equation}
then 
\begin{equation} \label{eq:mainconj3} \# \{ P \in \mathcal P_n^{\text{ind},A}~, \,\, \operatorname{Gal}(P/\mathbb Q) \notin \{A_n, S_n\} \} / \# \mathcal P_n^{\text{ind},A} \to 0~, \quad n \to \infty~. \end{equation}
In a joint work with Koukoulopoulos \cite{BSKK}, they relaxed the condition (\ref{eq:assumbsk}) to include, for example, intervals $A=\{1,\ldots, L\}$ of any length $L\geq 35$. If we allow $L$ to grow arbitrarily slowly with the degree $n$, then with probability tending to $1$ the Galois group is $S_n$ \cite{BSG}.

Breuillard and Varj\'u proved (\ref{eq:mainconj3}) (and thus also (\ref{eq:mainconj})) for any fixed finite $A \subset \mathbb Z \setminus \{0\}$ of cardinality $\geq 2$, conditionally on the extended Riemann Hypothesis (\ref{eq:rh}) for all polynomials with coefficients in $A$.

\medskip\noindent We conclude the discussion of the Galois group of polynomials in $\mathcal P_n^{\text{ind},A}$, with the classical van der Waerden theorem \cite{Waerden} which addresses the typical behaviour of the Galois group when $n$ is fixed and $A$ is the interval $\{-L,\ldots, L\}$ with $L \to \infty$: 
\begin{equation}\label{eq:vdW}
    \# \{ P \in \mathcal P_n^{\text{ind},\{-L,\ldots, L\}}~, \,\, \operatorname{Gal}(P/\mathbb Q) \neq S_n \} \, / \,  \# \mathcal P_n^{\text{ind},\{-L,\ldots, L\}}  \to 0~, \qquad L \to \infty~.
\end{equation}
There  is a vast literature of obtaining power saving upper bounds on the left-hand side of (\ref{eq:vdW}): see particularly \cite{Knobloch,Gallagher,Zywina,Dietmann,ChowDietmann1,Andersonetal,ChowDietmann2}. Recently, Bhargava \cite{Bhargava} proved that 
\[ \# \{ P \in \mathcal P_n^{\text{ind},\{-L,\ldots, L\}}~, \,\, \operatorname{Gal}(P/\mathbb Q) \neq S_n \}\, / \, \# \mathcal P_n^{\text{ind},\{-L,\ldots, L\}} \leq C_n/L~, \]
which is  best possible, up to the value of $C_n$.

\medskip
\noindent
Another natural family of polynomials arises from the theory of random matrices. Classically, one considers an algebraic group \( G \) over $\mathbb Z$ -- e.g.\ \( G = \SL_n \) -- and investigates the irreducibility and Galois group of a typical element of \( G(\mathbb{Z}) \). 
One natural model fixes a finite set of generators \( S \) and samples elements uniformly from a ball of large radius \( L \) in the Cayley graph of $G(\mathbb{Z})$. In this model, irreducibility (with probability tending to one) remains an open problem.

Rivin~\cite{Rivin} proposed a more tractable modification of this model: instead of sampling uniformly from a ball, one considers the \( \ell \)-th step of a random walk on the Cayley graph.  For $G = \SL_n$ and $\operatorname{Sp}_{2n}$, he demonstrated \cite{Rivin, Rivin2} that the characteristic polynomial of the resulting random matrix is irreducible, and, in the case of $\SL_n$, has Galois group \( S_n \), with probability \( 1 - C\exp(-c\ell) \).

This result was subsequently generalized to other connected algebraic groups by Jouve, Kowalski, and Zywina~\cite{JouveKowalskiZywina}, and to more general algebraic groups by Lubotzky and Rosenzweig~\cite{LubotzkyRosenzweig}. (Both of these works consider the more general setting of Zariski dense finitely generated subgroups of $G$, rather than just arithmetic subgroups.) In these extensions, the structure of the Galois group depends on both the Weyl group associated with \( G \) and, in the non-connected case, the connected component in which the walk lies. For example, for $G = \operatorname{Sp}_{2n}$, the typical Galois group is equal to $C_2 \wr S_n$.

Another possibility is to select elements uniformly from the set of
\begin{equation}\label{eq:gn} \left\{ \text{$n \times n$ matrices with entries in A}\right\} \cap G~.   \end{equation}
If $G \subset \operatorname{GL}_n$ is a connected semisimple algebraic group and $A = \{-L, \cdots, L\}$, $L \to \infty$,  Gorodnik and Nevo  \cite{GorodnickNevo} showed that, as $L \to \infty$, the Galois group is equal to the Weyl group of $G$  for most of such matrices. In particular, if $G = \SL_n$, the Galois group is typically equal to $S_n$.  This conclusion also holds for $G = \operatorname{GL}_n$, as follows from Hilbert's irreducibility theorem.

\medskip
\noindent

The complementary limiting regime is to fix $A$ and let $n \to \infty$. In this case, a matrix $M$ chosen at random in (\ref{eq:gn}) is invertible with probability exponentially close to one (see Kahn--Koml\'os--Szemer\'edi \cite{KKS} for $A = \{-1,+1\}$ and Rudelson--Vershynin \cite{RV} for the general case, as well as for more general distributions of the matrix entries), therefore the basic question is to study the properties of typical representatives of 
\[  
    \mathcal P_n^{\text{RM}, A} = \left\{ \text{characteristic polynomial of an $n \times n$ matrix  with entries in $A$} \right\}~, 
\]
where $A \subset \mathbb Z$, $2 \leq \#A < \infty$. Babai  and Vu--Wood   conjectured  (see \cite{V, Eb}) that irreducibility  (\ref{eq:mainconj}) holds in this case. Eberhard showed that (\ref{eq:mainconj3}), and hence (\ref{eq:mainconj}), holds under the extended Riemann Hypothesis (\ref{eq:rh}) for any such $A$, and unconditionally under the assumption (\ref{eq:assumbsk}) on the set $A$. These results build on the methods of \cite{BV,BSK}, respectively.

A closely related family consists of the characteristic polynomials of random symmetric matrices:
\[  
    \mathcal P_n^{\text{SRM}, A} = \left\{ \text{characteristic polynomial of a symmetric $n \times n$ matrix  with elements in $A$} \right\}~. 
\]
Eberhard conjectured that this family satisfies (\ref{eq:mainconj}) (irreducibility) and (\ref{eq:mainconj2}) (large Galois group); the former was proved by Ferber, Jain, Sah and Sawhney \cite{FJSS} conditionally on (\ref{eq:rh}), while the latter, and even  (\ref{eq:mainconj3}), is still open.

The results of \cite{Eb} and \cite{FJSS} naturally lead to the question whether (\ref{eq:mainconj}) and (\ref{eq:mainconj2}) hold for random matrices depending on fewer random variables. We focus on the family $H_n$; for the case of Bernoulli distribution $\mathbb P(V_1 = 0) = \mathbb P(V_1 = 1) = 1/2$  it  is of cardinality $2^n$, similarly to  $\mathcal P_n^{\text{ind}, \{0,1\}}$. We chose to focus on this particular family, which is arguably the simplest one, but other random matrices with few non-zero diagonals can be considered as well (see further below).

We also remark that the spectral properties of $H_n$ (and its infinite-volume limit) have been extensively studied in mathematical and theoretical physics under the name of (one-dimensional) Anderson model (see e.g.\ the monograph of Bougerol--Lacroix \cite{BL}, as well as the more recent monograph Aizenman--Warzel \cite{AW} presenting a broader circle of questions in the spectral theory of random operators). There, the focus is on the asymptotic, rather than algebraic and arithmetic, properties of the eigevalues (and eigenvectors) -- however, as we see below, our arguments make use of ideas developed in those fields, after appropriate adjustments.

\subsection{Several corollaries} We mention a couple of corollaries of Theorem~\ref{thm:1}. Recently, Emmrich showed \cite{Em} that if the characteristic polynomial $P$ of a matrix $H$ (say, with integer elements) satisfies $\operatorname{Gal}(P / \mathbb Q) \geq A_n$, then all the minors of any matrix $U$ diagonalising $H$ are non-zero. 
This conclusion was first proved by Chebotarev (see \cite[p.~274]{O}) for the adjacency matrix of a cycle graph (for which of course the Galois group is abelian -- the eigenvalues lie in a cyclotomic field -- hence much smaller). 
Combining this result of Emmrich with our theorem, we obtain, that, conditionally on (\ref{eq:rh}), all the minors of (any)  matrix $U_n$ diagonalising our $H_n$ are non-zero with probability $\geq1-C\exp(-cn)$.

It was proved by  Bir\'o (see \cite{Fr}),  Goldstein--Guralnick--Isaacs \cite{GGI}, Meshulam (unpublished) and Tao \cite{T}  that any matrix  $U$ the minors of which do not vanish satisfies the following uncertainty principle: for any non-zero vector $v \in \mathbb C^n \setminus \{0\}$,
\begin{equation}\label{eq:tao-unc}
\# \operatorname{supp} v +  \# \operatorname{supp} Uv \geq n+1 \end{equation}
(this inequality is sharp, since for any $1 \leq k \leq n$ one can find a linear combination of the first $k$ columns of $U$ which has at least $k-1$ zero co\"ordinates). 
Consequently, the matrices $U_n$ diagonalising $H_n$ satisfy  (\ref{eq:tao-unc}) with probability $\geq 1 - C \exp(-cn)$.

These properties could be contrasted with the results from the theory of Anderson localisation, which assert that the eigenvectors (the columns of $U_n$, normalised to have unit norm) are exponentially peaked: with probability $ \geq 1 - C \exp(- cn)$ one can order the eigenpairs so that 
\[ |(U_n)_{jk}| \leq C n^{0.001}  \exp(-c |j - k|)~,  \quad 1 \leq j,k \leq n~.\]
According to this estimate, for each vector $v$ with $\# \operatorname{supp} v = 1$ there exists an interval $I_v \subset [1, n]$ of length $\leq \frac{0.002}{c} \log n$, i.e.\ much smaller than $n = \# \operatorname{supp} Uv$, which carries most of the $\ell_2$ norm of $U_nv$ (the norm of $U_nv$ restricted to the complement of $I_v$ is bounded by a negative power of $n$).

\subsection{Extensions}

Consider the following extension of the model (\ref{eq:defHn}):
 \begin{equation}\label{eq:defHn+}
\tilde H_n = \left( \begin{array}{cccccc} 
V_1 & W_1 & 0 & 0 & 0 & \cdots \\
W_1 & V_2 & W_2 & 0 & 0 & \cdots \\
0 & W_2 & V_3 &W_3 & 0 & \cdots \\
\vdots&     &   \ddots & \ddots & \ddots &   \\ 
0 &&\cdots&W_{n-2} &V_{n-1} & W_{n-1}\\
0 && \cdots &0 & W_{n-1} & V_n \\ 
\end{array} \right)~, \end{equation}
where  $V_1, \cdots, V_n, W_1, \cdots, W_{n-1}$ are jointly independent, $V_j$ are sampled from one distribution on $\mathbb Z$, and $W_j$ from another one.

If the distribution of $V_1$ is non-degenerate, and both $V_1$ and $W_1$ satisfy a tail estimate of the form (\ref{eq:tail}), one can establish a conclusion similar to that of Theorem~\ref{thm:1}.

The situation is more complicated if $V_1\equiv a$ is constant. As we point out in Section \ref{s:dyson}, it is no loss to assume $a=0$, but here we state the results for general $a$. We start from the following observation:
\begin{cl}\label{cl:d-1} If $V_1 \equiv a$, then the spectrum of $\tilde H_n$ of (\ref{eq:defHn+}) is symmetric about $a$; in particular, the Galois group of  the characteristic polynomial $P_n(\lambda) = \det (\tilde H_n - \lambda \mathbbm{1})$ is contained in the wreath product $C_2 \wr S_{\lfloor n/2\rfloor} = C_2^{\lfloor n/2\rfloor} \rtimes S_{\lfloor n/2\rfloor}$  (viewed as a subgroup of $S_{2\lfloor n/2\rfloor} \leq S_n$) of the cyclic group $C_2 = \{-1, 1\}$ of order $2$ and the symmetric group, and $P_n$ is reducible for any odd $n \geq 3$.
\end{cl}
 In order to prove the claim, it suffices to note that conjugation by the diagonal matrix $\operatorname{diag}(-1, 1, -1, 1, \cdots)$ takes $\tilde H_n$ to $2a \mathbbm{1} - \tilde H_n$.

Next, we record two cases in which the Galois group is (much) smaller than $C_2 \wr S_{\lfloor n/2\rfloor}$ (the first statement follows from the recursion (\ref{eq:d-rec}) (see Section~\ref{s:dyson}); the second one follows from the block structure of $\tilde H_n$ in case some of the $W_j$ vanish).
\begin{cl} Consider the model (\ref{eq:defHn+}) with $V_1 \equiv a$.
\begin{enumerate}
\item If $|W_1| \equiv b > 0$, then $P_n(x) = b^n U_n((x-a)/(2b))$ is reducible for any $n \geq 2$, and, moreover, the Galois group of $P_n$ is abelian;
\item if  $\mathbb P (|W_1| = 0) = \tau >0$, then for any $1 \leq k \leq n-1$
\[ \mathbb P \{ \text{$P_n$ splits into $ \geq k+1$ irreducible factors} \} \geq 1 - \binom{n-1}{k-1} (1-\tau)^{n-k}\]
(in particular, the number of irreducible factors is linear in $n$, with probability exponentially close to one).
\end{enumerate}
\end{cl}

 In Section~\ref{s:dyson} we prove the following theorem, which asserts that once these cases are excluded, the Galois group is close to $C_2 \wr S_{\lfloor n/2\rfloor}$.  For a positive integer $m$,  let

\begin{equation}\label{eq:def-Um}
\u[m] = \left\{ a \in C_2^m \, : \, \prod_{j=1}^m a_j = 1\right\}~.
\end{equation}
For $m\ge 2$, $\u[m] \rtimes A_m$ is the derived subgroup of $C_2\wr S_m = C_2^m \rtimes S_m$, and has index $4$ therein.

\begin{thm}\label{thm:2} Consider the characteristic polynomial $P_n(\lambda)$ of the model (\ref{eq:defHn+}) with $V_1 \equiv a$. If $|W_j|$ has a non-degenerate distribution supported on $\mathbb Z_{> 0}$ and 
\begin{equation}\label{eq:tail2}\mathbb P\{ |W_1| \geq k \} \leq B \log^{-b} k~, \quad k \geq 3 \end{equation}
for some $B, b > 0$, then, assuming the extended Riemann hypothesis (\ref{eq:rh}), 
\[ \mathbb P \left\{ \operatorname{Gal}(P_n / \mathbb Q) \geq  \u[{\lfloor n/2\rfloor}] \rtimes A_{\lfloor n/2\rfloor} \right\} \geq 1 - C \exp(-cn)~. \]
 In particular, for even $n$ the polynomial is irreducible with probability exponentially close to one; for odd $n$ it is a product of a linear term and an irreducible one.
\end{thm}
We conjecture that, under the assumptions of the theorem, $\operatorname{Gal}(P_n / \mathbb Q)  = C_2 \wr S_{\lfloor n/2\rfloor}$ with probability exponentially close to one.

Finally, we mention that the atypical behaviour of the spectrum of (\ref{eq:defHn+}) with $V_j \equiv a$ around $a$ is  known in spectral theory as Dyson's spike (see Dyson \cite{Dyson} as well as the recent work of Kotowski and Vir\'ag \cite{KV} and references therein). Both the spectral-theoretic anomaly (the vanishing of the Lyapunov exponent at $\lambda = a$) and the algebraic one (the small Galois group of $P_n$) are consequences of the symmetry of the spectrum about the origin (which also implies that the transfer matrices at $\lambda = a$ lie in the triangular subgroup, a soluble subgroup of $\SL_2$).

\paragraph{Acknowledgement} We are grateful to Alex Gamburd and Elon Lindenstrauss   for referring us to the works \cite{BG,GS} which led to an improvement in the estimates on the probability in the main results. We also thank Alex Lubotzky and Peter Sarnak for their helpful comments.

The first author is supported by the Israel Science Foundation (grant no.\ 366/23). The third author is partially supported by the Philip Leverhulme Prize (PLP-2020-064); he is also grateful to  Gady Kozma and the Morris Belkin Visiting Professor programme for the hospitality at the Weizmann Institute of Science, where parts of this work have been performed.

\section{Overview of the proof}

The proof is based on the general strategy of Breuillard and Varj\'u \cite{BV}. Let $P \in \mathbb Z[x]$ be a polynomial of degree $\leq n$.  Let $Z(P)$ be the set of (distinct) complex zeros of $P$. Also, for a set $\Omega$ and a natural number $k$, we denote by $(\Omega)_k$ the set of ordered $k$-tuples of distinct elements of $\Omega$.

In the sequel, we shall abuse notation and denote by the same letter $P$ the reduction of $P$ modulo a prime $p$. Let $Z(P, p)$ be the set of (distinct) roots of $P$ in $\mathbb F_p$.  The following estimate of \cite{BV}, relying on the Riemann hypothesis (\ref{eq:rh}), relates the number of orbits of the action of the Galois group   $\operatorname{Gal}(P / \mathbb Q)$  on the set  $(Z(P))_k$   to a weighted average of the cardinalities of $(Z(P, p))_k$ as $p$ varies over primes lying in a long interval: 
\begin{equation}\label{eq:bv}
\left| \frac{1}{x} \sum_{x < p \leq 2x} \log p \times \#  (Z(P, p))_k - \#\left( (Z(P))_k / \operatorname{Gal}(P / \mathbb Q)\right) \right| \leq C k n^k   \log (H(P) n) \, \frac{\log^2 x}{x^{1/2}}~.
\end{equation}

Here  $H(P)$ is the height of $P$ (the maximal absolute value of the coefficients), and the sum is over the primes in the interval $(x, 2x]$. The inequality (\ref{eq:bv}) follows from the combination of \cite[Proposition~9]{BV} and \cite[Lemma 21]{BV}. Note that we work with falling factorials, whereas \cite{BV} with powers; this is mainly a matter of notational convenience.

To analyse the average of $\# (Z(P, p))_k$ for polynomials with random independent coefficients, Breuillard and Varj\'u study the random walk on the affine group modulo $p$. The properties of our random polynomial $P_n$ are closely connected to a random walk on $\mathrm{SL}_2(p)$. Formally, denote
\[ T(\lambda) = \left( \begin{array}{cc} \lambda & - 1 \\ 1 & 0 \end{array}\right),  \quad \Phi_n(\lambda) = T(\lambda - V_n) T (\lambda-V_{n-1}) \cdots T (\lambda - V_1)~, \]
with the convention $\Phi_0(\lambda) = \mathbbm{1}$.
We have the following deterministic identity (well known in the spectral theory of one-dimensional Schr\"odinger operators, where $\Phi_n(\lambda)$ are known as transfer matrices).

\begin{lemma} \label{lem:computecharpolyfromtransfer}
$P_n(\lambda) = (\Phi_n(\lambda))_{11}$ (deterministically for any matrix of the form (\ref{eq:defHn})).
\end{lemma}

\begin{proof}
The identity is verified directly for $n=0$. For the induction step, observe that by row expansion
\[ P_n(\lambda) = (\lambda - V_n) P_{n-1}(\lambda) - P_{n-2}(\lambda)~, \quad n\geq 1~,\] 
with the convention $P_{-1} \equiv 0$, whence
\[ \binom{P_{n}(\lambda)}{P_{n-1}(\lambda)} = T (\lambda - V_n) \binom{P_{n-1}(\lambda)}{P_{n-2}(\lambda)}~. \qedhere\]
\end{proof}

\begin{cor}\label{cor:height} Assume the tail bound (\ref{eq:tail}). Then   
\[ \mathbb P \{ \log H(P_n) \geq t \} \leq 2^b B n^{1+b} t^{-b}~, \quad t \geq 2n~,
\]
and, in particular, for any $\epsilon > 0$
\begin{equation}\label{eq:height}\begin{split}
 \mathbb P \left\{\log H(P_n) \geq \exp(\epsilon n) \right\}  &\leq C_\epsilon \exp(-(b/2) \epsilon n) \\
 \mathbb P \left\{\log H(P_n) \geq \exp(n^\epsilon) \right\}  &\leq C_\epsilon \exp(-(b/2) n^\epsilon)~. \end{split}\end{equation}
\end{cor}
\begin{proof}
The tail bound  (\ref{eq:tail}) implies that
\[ \mathbb P \left\{ |V_j| \geq  \exp(t / (2n)) \right\}
\leq B (t/(2n))^{-b}~,  \quad t \geq 2n~,
 \]
whence 
\[ \mathbb P \left\{ \max_j |V_j| \geq  \exp(t/(2n))  \right\} \leq B n (t/(2n))^{-b} = 2^b B n^{1+b} t^{-b}~.\]
On the complementary event, we have by the Cauchy formula 
\[ \log H(P_n) \leq \max_{|z|=1} \log \| \Phi_n(z)\| \leq  
n \log(2 + \exp(t/(2n))) \leq t~. \qedhere\]
\end{proof}

As noted above, random walks in $\SL_2(p)$ arise naturally in our problem: for each $\lambda$ the reduction of the sequence $\Phi_n(\lambda)$ modulo a prime $p$ forms a random walk in $\SL_2(p)$. It will be convenient (though far from crucial) to work with the projection of this walk to $\PSL_2(p)$. Abusing notation, we denote the image of $\Phi_n(\lambda)$ in both $\SL_2(p)$ and $\PSL_2(p)$ by the same letters $\Phi_n(\lambda)$, i.e.\  denote by $\Phi_n(\lambda) \in \SL_2(p)$ or $\Phi_n(\lambda) \in \PSL_2(p)$, $\lambda \in \mathbb F_p$,   the image of a matrix $\Phi_n(\tilde \lambda)$, where $\tilde\lambda \in \mathbb Z$ is such that $\tilde \lambda \equiv \lambda \mod p$. Recall that (for $p\ge 3$)
\[ N_p \overset{\textrm{def}}{=} \# \PSL_2(p) = p(p^2-1)/2~.\]

\medskip
\noindent We first explain how to establish the conclusion of Theorem~\ref{thm:1} with a weaker, stretched-exponential bound on the probability, $1 - C \exp(-n^c)$ for an implicit constant $c>0$. This argument will use relatively classical results; the main non-elementary input will come from the work of Helfgott \cite{H}. Then we shall explain how to  improve the probability bound to exponential, as claimed in the theorem. This will require to use the classification of finite simple groups, as well as the works of Bourgain--Gamburd--Sarnak \cite{BGS}, Breuillard--Gamburd \cite{BG} and of Golsefidy--Srinivas \cite{GS}.

\subsection{The stretched exponential bound}
 In Section~\ref{s:mix}, we shall prove
\begin{prop}\label{prop:mix} There exists $C > 0$ (depending on the quantity (\ref{eq:atombd})) such that the following holds. Let $p_1 \geq \cdots \geq p_k \geq C$ be prime numbers and let $n \geq C k^7 \log^{C}p_1$. Then, for any $\lambda_1 \in \mathbb F_{p_1}, \ldots, \lambda_k \in \mathbb F_{p_k}$ such that $\lambda_i \neq \lambda_{j}$ whenever $p_i = p_j$ and $i \neq j$, we have 
\[
\begin{split} 
    d_k(n) &\overset{\textrm{def}}{=} \sum_{(g_1, \ldots, g_k) \in \PSL_2(p_1)\times \cdots \times \PSL_2(p_k)} \left| \mathbb P \left\{ \forall 1 \leq j \leq k \, : \, \Phi_n(\lambda_j) = g_j \right\} - \prod_{j=1}^k N_{p_j}^{-1} \right| 
    \leq \exp(-\frac{n} {C k^6 \log^{C}p_1} )~.
\end{split}
\] 
\end{prop}

\noindent Then we shall deduce:
\begin{cor}\label{cor:mix} There exists $C > 0$ so that for $x \geq \max\{ 5, k!\}$, $n \geq  C k^8 \log^{ C } x$,
\begin{equation}\label{eq:cor:mix} \mathbb P \left\{ \left| \frac{1}{x}  \sum_{x < p \leq 2x} \log p \times  \# ( Z(P_n, p))_k  - 1 \right| \geq \frac{1}{2} \right\}  
\leq  C \left[ \exp(-\frac{n} { C  k^6 \log^{ C}x})+\frac{k^{k} \log x}{x} \right] ~. \end{equation}
\end{cor}

\noindent This will suffice for the
\begin{proof}[Proof of Theorem~\ref{thm:1} with a stretched exponential bound on the probability]
Let $C$ be the constant of Corollary~\ref{cor:mix}, let  $c>0$ be a  sufficiently small constant, and let $n$ be sufficiently large. For $k \leq n^c$, let $x =\exp((n/ k^8)^{1/(2C)})$. 
Since $k!\leq k^{k+1}/e^{k-1}$, we have $x\geq \max\{5,k!\}$.  For this choice of $x$ and $k$, the right-hand side of (\ref{eq:cor:mix}) is bounded by $\exp(-n^c)$. 
The right-hand side of (\ref{eq:bv}) is also bounded by $\exp(-n^c)$, with probability  $\geq 1 - C \exp(-c n^c)$ (recall the second statement of Corollary~\ref{cor:height}). Thus \eqref{eq:bv} and Corollary~\ref{cor:mix} imply that on an event of probability $\geq 1 - \exp(-n^c)$,
\[ 
    |\# \left( (Z(P_n))_k / \operatorname{Gal}(P / \mathbb Q)\right) - 1 | \leq 1/2~, 
\]
whence $\operatorname{Gal}(P/ \mathbb Q)$ acts $k$-transitively on $Z(P_n)$. In particular (taking $k=1$), $P_n = Q^m$ for some irreducible $Q \in \mathbb Z[x]$. Lemma~\ref{l:mult} below ensures that, with probability $\geq 1 - C \exp(-n^c)$, this could only hold for $m=1$, thus $P$ is irreducible and in particular $\# P_n^{-1}(0) = n$.

To conclude the proof of the theorem, we recall (see \cite{BV} for references) that a group acting $k$-transitively on a set of size $n$ is either $S_n$ or $A_n$, provided that $k \geq 6$ (using the classification of finite simple groups), or that $k \geq 6 \log n$ (by an argument of Wielandt avoiding classification), or that $k\geq 30 \log^2 n$ (using an elementary argument of Bochert and Jordan from the XIX  century).  
\end{proof}

\noindent The lemma that we used to rule out proper powers of irreducibles is as follows:
\begin{lemma}\label{l:mult} There exists $C, c > 0$ so that the probability that there exists $m \geq 2$ and $Q \in \mathbb Z[x]$ such that $P_n = Q^m$ is bounded by $C \exp(-n^c)$.
\end{lemma}
\noindent We prove it in Section~\ref{s:mult}.

\subsection{The exponential bound} Now we explain how to upgrade the arguments above to obtain the statement as claimed. In Section~\ref{s:upgr}, we use the results of Breuillard--Gamburd \cite{BG} and Golsefidy--Srinivas \cite{GS} to obtain the following version of Corollary~\ref{cor:mix}:

\begin{lemma}\label{cor:mix-upgr} For any $k \geq 1$ there exist $C_k, c_k>0$ such that for $c_k x^{\frac{1}{2k}} \geq n \geq C_k \log x$
\begin{equation}\label{eq:cor:mix-upgr} \mathbb P \left\{ \left| \frac{1}{x}  \sum_{x < p \leq 2x} \log p \times \# ( Z(P_n, p))_k  - 1 \right| \geq \frac{1}{2} \right\}  
\leq  C_k \left[ \exp(-c_k n)+\frac{\log x}{x} \right] ~. \end{equation}
\end{lemma}

\noindent In the second part of Section~\ref{s:mult}, we improve the probability estimate in Lemma~\ref{l:mult}:
\begin{lemma}\label{l:mult-upgr} There exist $C, c > 0$ so that the probability that there exist $m \geq 2$ and $Q \in \mathbb Z[x]$ such that $P_n = Q^m$ is bounded by $C \exp(-c n)$.
\end{lemma}

\begin{proof}[Proof of Theorem~\ref{thm:1} with an exponential bound on the probability]
Arguing as above with Lemma~\ref{cor:mix-upgr} in place of Corollary~\ref{cor:mix}, taking $x = \exp(c'n)$ in (\ref{eq:bv}) and using the first estimate of (\ref{eq:height}) in Corollary~\ref{cor:height},  we obtain that
\[ \mathbb P \left\{ \# ((Z(P_n))_k / \operatorname{Gal}(P_n, \mathbb Q)) = 1 \right\}  \geq 1 - C_k \exp(- c_k n)~.\]
Applying this conclusion with $k=1$ and using  Lemma~\ref{l:mult-upgr} in place of Lemma~\ref{l:mult}, we obtain that $P_n$ is irreducible with probability $\geq 1 - C \exp(-cn)$. Now applying the conclusion with $k = 6$ and using the aforecited corollary of the classification of finite simple groups, we obtain that $\operatorname{Gal}(P_n / \mathbb Q) \geq A_n$ with probability $\geq 1 - C \exp(-cn)$. 
\end{proof}

\section{Algebraic preliminaries}\label{s:alg}

Here we prove two auxiliary results required in the next section. 
The first lemma is a finite field analogue of a fact  known over $\mathbb C$  (in which case the matrices generate a Zariski dense subgroup of $\operatorname{PGL}_2(\mathbb C)$ -- see further Section~\ref{s:mult} below).  
\begin{lemma}\label{l:gen} Let $v, v'$ be two distinct integers; let $p  \geq C(v,v') =\max(5, |v - v'|+1)$  be a prime number, and let $\lambda \in \mathbb F_p$. Then $\langle S^2S^{-2}\rangle = \langle S^{-2}S^{2}\rangle = \PSL_2(p)$, where  $S= \{ T(\lambda-v), T(\lambda-v')\}$.
\end{lemma}
Here and afterwards, we denote by $S^2S^{-2}$ the set of elements of the form $g_1g_2g_3^{-1}g_4^{-1}$, where $g_i\in S$, and by $\langle S^{-2}S^{2}\rangle$ the set of elements of the form $g_1^{-1}g_2^{-1}g_3g_4$.
Similarly to the complex case, it is easier to check that $S$ generates the group; however, it is important for the sequel to have a symmetric set of generators of the form $S^{j}S^{-j}$ (cf.\ Goldsheid \cite{G1} for the use of similar generating sets in the complex setting, for a somewhat similar purpose). It is not true that  $SS^{-1}$ generates $\PSL_2(p)$ -- indeed $\langle SS^{-1}\rangle$ is cyclic, see (\ref{eq:a*b'}).

\begin{proof}
Observe that for any $X \subset G$ the subgroups $\langle X X^{-1} \rangle$ and $\langle X^{-1} X \rangle$  are conjugate: indeed, assuming without loss of generality that $X \neq \varnothing$ and picking $x \in X$, $x^{-1} X X^{-1} x \subset \langle X^{-1} X \rangle$, whence $\langle X^{-1} X \rangle$ contains a conjugate of $\langle X X^{-1} \rangle$, and vice versa. Thus the two subgroups are of the same size, and thus they are conjugate.

According to this observation, it suffices to show that $G = \langle S^2S^{-2}\rangle = \PSL_2(p)$. If this is false, $G$ is contained in one of the maximal subgroups of $\PSL_2(p)$. According to a theorem of Moore and Wiman, as presented by Dickson \cite{Dickson}, these have the following form (the implied action is on the projective line):
\begin{enumerate}
\item dihedral groups of order $p-1$ for $p\geq 13$: each is the stabilizer of an unordered pair of points; 
\item dihedral groups of order $p + 1$ for $p \neq 7, 9$;
\item groups of order $p(p-1)/2$; each is the stabilizer of a point;
\item $S_4$ when $p\equiv\pm1 \mod8$;
\item $A_4$ when $p \equiv \pm 3, 5, \pm 13\mod{40}$. 
\item $A_5$ when $p\equiv\pm1\mod10$.
\end{enumerate}
Observe that for $p$ large enough (greater than $|v'-v|$)
\begin{equation}\label{eq:a*b'} u=T(\lambda-v) T(\lambda-v) T(\lambda-v)^{-1} T(\lambda-v')^{-1}  = T(\lambda-v) T(\lambda-v')^{-1} = \left( \begin{array}{cc} 1 &  v'-v \\ 0 & 1 \end{array}\right)~\end{equation}
has order $p$, which rules out the cases (1), (2), (4), (5), and (6) (note that if $p\equiv \pm1 \mod 10$, then $p\geq 11$). Further, $u$ fixes only the point $\infty = [ 1:0]$, whereas the element 
\[ T(\lambda-v) T(\lambda-v) T(\lambda-v')^{-1} T(\lambda-v)^{-1}   \]
fixes only $\lambda-v=[\lambda-v:1]=T(\lambda-v)\infty \neq \infty$; this rules out the case (3). 
\end{proof}

\begin{lemma} \label{l:gen_prod} Let $v, v'$ be two distinct integers; let $p_1 \geq \cdots \geq p_k \geq C(v, v') = \max(5, |v - v'|+1)$ be  prime numbers, and let $\lambda_1 \in \mathbb F_{p_1}, \cdots, \lambda_k \in \mathbb F_{p_k}$ be such that $\lambda_i \neq \lambda_{j}$ whenever $p_i = p_j$ and $i \neq j$.
Then $\langle S^3S^{-3}\rangle = \langle S^{-3}S^{3}\rangle= \PSL_2(p_1) \times \PSL_2(p_2) \times \cdots \times \PSL_2(p_k)$, where 
\begin{equation}\label{eq:gjk} S= \{(T(\lambda_1-v), \cdots, T(\lambda_k-v)),  (T(\lambda_1-v'), \cdots, T(\lambda_k-v'))\}~. \end{equation}  
\end{lemma}

We note that, in the case $k=2$ and $p_1=p_2=p$, it is not true that $\langle S^2S^{-2} \rangle=  \PSL_2(p)^2$; that is why we consider $S^3S^{-3}$.

\begin{proof}
The case $k =1$ is covered by the previous lemma.

For $k = 2$, we use the following corollary of Goursat's lemma \cite[Section 4.3]{Sc}:  
if $\Gamma_1, \Gamma_2$ are non-abelian finite simple groups, and $H$ is a subgroup of $\Gamma_1 \times \Gamma_2$ that projects onto  both $\Gamma_1$ and $ \Gamma_2$, then either $H = \Gamma_1 \times \Gamma_2$, or $H$ is a diagonal; that is, there exists an isomorphism $\phi\colon \Gamma_1 \to \Gamma_2$ such that $H =  \{(g, \phi(g)) \, : \, g \in \Gamma_1\}$.   We take $\Gamma_j = \PSL_2(p_j)$. By Lemma~\ref{l:gen}, the group $H$ generated by $S^3S^{-3}$ surjects onto each of the factors. If $H =  \PSL_2(p_1) \times  \PSL_2(p_2)$, we are done. Otherwise, $H$ is a diagonal. In particular, $p_1 = p_2 = p$.

Observe that all the automorphisms of $\PSL_2(p)$ have the form $g \mapsto h g h^{-1}$ for $h \in \operatorname{PGL}_2(p)$. Assume that for every $a = (a_1, \cdots, a_6) \in \{v,v'\}^6$
\[\begin{split} 
& T(\lambda_2 - a_1) T(\lambda_2-a_2) T(\lambda_2 - a_3) T(\lambda_2 - a_4)^{-1}   T(\lambda_2 - a_5)^{-1}  T(\lambda_2 - a_6)^{-1}  \\
&\quad= 
h T(\lambda_1 - a_1) T(\lambda_1-a_2) T(\lambda_1 - a_3) T(\lambda_1 - a_4)^{-1}   T(\lambda_1 - a_5)^{-1}  T(\lambda_1 - a_6)^{-1}h^{-1}. \end{split}\]
(This is an equality in $\PSL_2(p)$.) Taking $a = (v,v,v,v,v,v')$, we have (in $\PSL_2(p)$):
\[  \left( \begin{array}{cc} 1 & v'-v \\ 0 & 1 \end{array} \right)  = h  \left( \begin{array}{cc} 1 & v'-v \\ 0 & 1 \end{array} \right)  h^{-1}~;\]
it is clear that the equality also holds in $\SL_2(p)$, whence 
\[ h =  \left( \begin{array}{cc} 1 & x \\ 0 & 1 \end{array} \right)  \]
for some $x \in \mathbb F_p$.  Now, for $a = (v,v,v,v,v',v)$
\[T(\lambda  - a_1) T(\lambda -a_2) T(\lambda  - a_3) T(\lambda  - a_4)^{-1}   T(\lambda  - a_5)^{-1}  T(\lambda  - a_6)^{-1} = T(\lambda-v)  \left( \begin{array}{cc} 1 & v'-v \\ 0 & 1 \end{array} \right) T(\lambda-v)^{-1}\]
fixes only $\lambda-v = [\lambda-v:1]$, hence $h$ must map $\lambda_1-v$ to $\lambda_2-v$, i.e.\ $x = \lambda_2 - \lambda_1$. On the other hand, for $a = (v,v,v,v',v,v)$
\[T(\lambda  - a_1) T(\lambda -a_2) T(\lambda  - a_3) T(\lambda  - a_4)^{-1}   T(\lambda  - a_5)^{-1}  T(\lambda  - a_6)^{-1} \]
fixes only the point $(T(\lambda-v)) (\lambda-v) = \lambda - v - 1/(\lambda-v)$, whence 
\[ x = (\lambda_2-v - 1/(\lambda_2-v)) - (\lambda_1-v - 1/(\lambda_1-v))~.\]
 Thus $\lambda_1 = \lambda_2$, in contradiction with the assumption.

Finally for $k \geq 3$, we observe that by what we have just proved the projection of the group $H$ generated by $S^3S^{-3}$ to any pair of coordinates $(i,j)$ is the full group $\PSL_2(p_i) \times \PSL_2(p_j)$. This implies, by another corollary of Goursat's lemma, that 
\[ H = \PSL_2(p_1) \times \PSL_2(p_2) \times \cdots \times \PSL_2(p_k)~.\qedhere\]
\end{proof}

\section{Mixing of the Markov chain on $\PSL_2(p)$}\label{s:mix}
\begin{proof}[Proof of Proposition~\ref{prop:mix}] Recall the following definition. If $G$ is a finite group, the diameter $\Delta(G)$ is the maximal diameter of the Cayley graph of $G$, i.e.\ the maximum of
\[ 
    \max_{g \in G} \min \{ k  \geq 0 \, : \, \exists s_1, \ldots, s_k \in S~, \, g = s_1 \cdots s_k\} 
\]
over all the symmetric generating sets $S$ of $G$ (with the convention that the empty product is equal to $1$).
According to a result of Helfgott \cite[Main Theorem]{H}, $\Delta(\operatorname{SL}_2(p)) \leq C \log^C p$, 
thus the same is also true for $\PSL_2(p)$. According to a result of Babai and Seress \cite[Lemma 5.4]{BS}, we have (using the fact that $\PSL_2(p)$ is simple, non-abelian for $p \geq 5$)
\[ \Delta = \Delta(\PSL_2(p_1) \times \PSL_2(p_2) \times \cdots \times \PSL_2(p_k)) \leq 20 k^3 \max_j \Delta(\operatorname{SL}_2(p_j))^2 \leq  20 C^2 k^3 \log^{2C} p_1~. \]
We consider four Markov chains on $\PSL_2(p_1) \times \PSL_2(p_2) \times \cdots \times \PSL_2(p_k)$. All the four have transitions from 
$h$ to $gh$,  where: $g = (T(\lambda_1 - V), \cdots, T(\lambda_k - V))$, where $V$  is chosen according to the distribution of $V_1$ (Chain 1);  $g = g_1 g_2 g_3$, where $g_j$ are chosen   independently from the same distribution as for the first chain (Chain 2); $g = g_1^{-1} g_2^{-1} g_3^{-1} g_4 g_5 g_6$,  where $g_j$ are chosen   independently from the same distribution (Chain 3); and $g = g_1^{-1} g_2^{-1} g_3^{-1} g_4 g_5 g_6$, where $g_j$ are chosen independently and uniformly in
\[ S = \{(T(\lambda_1-v), \cdots, T(\lambda_k-v)),  (T(\lambda_1-v'), \cdots, T(\lambda_k-v'))\}~, \]
where $v,v'$ are in the support of $V_1$ achieving the bound (\ref{eq:atombd}) (Chain 4). Denote the transition matrices of the four chains by $\Pi_j$, $1 \leq j \leq 4$. Observe that   $\Pi_2 = \Pi_1^3$ and $\Pi_3 = (\Pi_1^3)^t \Pi_1^3$; in particular, Chain 3 is reversible, and so is Chain 4.

Similarly to \cite{H}, we make use of the results of Diaconis and Saloff-Coste \cite[Section 3]{DSC}, that imply that the eigenvalues of $\Pi_4$, except for the top one, are bounded in absolute value by $1 - 1/(64 \Delta^2)$, i.e.\ 
\begin{equation}\label{eq:dsc}
\| \Pi_4 w \| \leq (1 - 1/(64 \Delta^2)) \|w\|
\end{equation}
 for any vector $w$ perpendicular to constants.

Now represent $\Pi_3 = \alpha \Pi_4 + (1-\alpha) \Pi'$, where $\alpha = \min(\mathbb P(V = v), \mathbb P(V = v'))$, and $\Pi'$ is another symmetric stochastic matrix (the definition of $\alpha$ ensures that the entries of $\Pi'$ are non-negative).  Let $w$ be any unit vector perpendicular to constants. Then
\begin{align*}
\| \Pi_3 w \| &\le \alpha\|\Pi_4 w\| +  (1-\alpha)\|\Pi'w\| \\ 
&\le \alpha(1-\frac{1}{64\Delta^2}) + 1-\alpha  = 1- \frac{\alpha}{64\Delta^2} \\ &\le \exp(- \alpha/ (64 \Delta^2))~.
\end{align*}
Consequently,  
\[ \| \Pi_2 w\| = \sqrt{\langle \Pi_2 w, \Pi_2 w \rangle} \leq \sqrt{\| \Pi_2^t \Pi_2 w \|} \leq \exp(- \alpha/ (128 \Delta^2))~,\]
and, iterating this inequality,
\[ \| \Pi_2^n w \| \leq \exp(- \alpha n / (128 \Delta^2))~, \]
and finally,
\[ \| \Pi_1^n w \| \leq \|\Pi_1\|^2  \| \Pi_2^{\lfloor n/3 \rfloor} w \| \leq \exp(- \alpha \lfloor n/3\rfloor / (128  \Delta^2))  
\leq  \exp(- \alpha n / (500  \Delta^2)) \]
(for $n \geq 9$). Thus
\[d_k(n) \leq C_1 N_p^{k/2} \exp(- \frac{\alpha n}{500 \Delta^2}) \leq 
C_2 p^{3k/2} \exp(-  \frac{n}{C_2 k^6 \log^{4C} p_1}) \leq \exp(-  \frac{n}{2C_2 k^6 \log^{4C} p_1}) \]
for $n \geq C_3k^7 \log^{4C+1} p_1$, with the constants depending only on $\alpha$.
\end{proof}

\begin{proof}[Proof of Corollary~\ref{cor:mix}]
Since for a uniform $g\in \PSL_2(p)$ the probability that $g_{11} \equiv 0\mod p$ is $1/(p+1)$, Proposition~\ref{prop:mix} implies that for $p \geq 5$, $n \geq C k^7 \log^C p$
\[
\begin{split}
&\left| \mathbb E \# ( Z(P_n, p))_k - 1  \right| \leq\left| \mathbb E  \# (Z(P_n, p))_k- \frac{(p)_k}{(p+1)^k}  \right|  + \left|\frac{(p)_k}{(p+1)^k} -1\right| \\
&\qquad\leq 
 \mathbb E \sum_{\substack{\lambda_1, \cdots, \lambda_k \\ \text{distinct}}} \left| \mathbb P \left\{ \forall 1 \leq j \leq k \, : \, 
P_n(\lambda_j) \equiv 0 \mod p \right\}  - \frac{1}{(p+1)^k} \right| + \frac{Ck^2}{p} \\
&\qquad \leq p^k \exp\left(-\frac{n} {C k^6 \log^{C}p}\right)  + \frac{Ck^2}{p} \leq \exp\left(-\frac{n} {C' k^6 \log^{C}p}\right)+ \frac{Ck^2}{p}~.
\end{split}
\]
Moreover, for $p >p' \geq 5$  
\[\begin{split} &\left|  \mathbb E \#  (Z(P_n, p))_k   \# (Z(P_n, p'))_k - 1  \right| \\
&\quad\leq  \sum_{\substack{\lambda_1, \cdots, \lambda_k \in \mathbb F_p\\ \text{distinct}}} 
\sum_{\substack{\lambda_1', \cdots, \lambda_k' \in \mathbb F_{p'} \\ \text{distinct}}} 
\left| \mathbb P \left\{ \forall 1 \leq j \leq k \, : \, 
P_n(\lambda_j) \equiv 0 \mod p,  P_n(\lambda_j') \equiv 0 \mod p'\right\} -  \frac{1}{(p+1)^{k}(p'+1)^k} \right| \\
&\qquad\qquad+ \frac{Ck^2}{2p'}\\
&\quad\leq p^{2k} \exp(- \frac{n}{C k^{6} \log^C p}) + \frac{Ck^2}{p'} \leq \exp(- \frac{n}{C' k^{6} \log^C p}) + \frac{Ck^2}{p'}~,
\end{split}\]
whence (by the previous estimate)  the same bound holds for $\left|  \mathbb E \#  (Z(P_n, p))_k   \big(\# (Z(P_n, p'))_k - 1\big)\right|$.

Further, let $A(k) = \sum_{\ell=0}^k \binom{k}{\ell}^2 \, \ell! \leq C k^k$, then, using the identity $(N)_k^2 = \sum_{\ell = 0}^k \binom{k}{\ell}^2 \ell! (N)_{2k-\ell}$, where now $(N)_k$ stands for the falling factorial $(N)_k = n! / (N-k)!$, we have for $n \geq C' k^8 \log^C p$:
\[\begin{split}
&\left| \mathbb E  ( \# (Z(P_n, p))_k)^2- A(k) \right| \\
&\qquad\leq \sum_{\ell=0}^k \binom{k}{\ell}^2 \, \ell! \, \left| \mathbb E  \# (Z(P_n, p))_{2k-\ell} - 1 \right|  \\
&\qquad \leq A(k) \left( \exp(-\frac{n} {C' k^6 \log^{C}p})+ \frac{Ck^2}{p} \right) 
\leq \exp(-\frac{n} {C'' k^6 \log^{C}p}) + \frac{C''k^{k+2}}{p}~.
\end{split}\]
As usual, the  Chebyshev function is defined as
\[ 
    \vartheta(x) = \sum_{\text{prime $p \leq x$}} \log p = (1+o(1)) x~.
\]
Then
\[\begin{split}
&  \mathbb E \frac{1}{(\vartheta(2x) - \vartheta(x))^2}  \left| \sum_{x< p \leq 2x} \log p \times \big( \# ( Z(P_n, p))_k - 1\big)  \right|^2 \\
&\qquad\leq \frac{C_1 \log^2 x}{x^2}  \Big\{
\sum_{x< p \leq 2x} \mathbb E |\# (Z(P_n, p))_k - 1|^2 \\
&\qquad\qquad+ 2 \sum_{x< p' < p  \leq 2x} \mathbb E (\# (Z(P_n, p))_k - 1)( \# (Z(P_n, p'))_k - 1) 
\Big\} \\
&\qquad\leq \frac{C_1 \log^2 x}{x^2}  \Big\{
\sum_{x< p \leq 2x} \left[ A(k) +  \left| \mathbb E (\# ( Z(P_n, p))_k)^2 - A(k) \right| + 2   \left| \mathbb E \#  (Z(P_n, p))_k - 1 \right| + 1  \right] \\
&\qquad\qquad+ 2 \sum_{x< p' < p  \leq 2x}\left|  \mathbb E \# (Z(P_n, p))_k  \big( \# (Z(P_n, p'))_k - 1\big) \right|
\Big\}  \\
&\qquad\qquad+ \frac{2C_1 \log x}{x} \sum_{x < p \leq 2x} \left| \mathbb E \# (Z(P_n, p))_k -1\right|~. 
\end{split}\]
For $n \geq C'' k^8 \log^C x$ and $x \geq \max(5, k!)$, we can use the bounds obtained above to continue the inequality as follows:
\[ \begin{split} 
&\qquad\leq \frac{C_2 \log^2 x}{x^2}  \Big\{
\sum_{x< p \leq 2x} \left[ k^k +  \exp(-\frac{n} {C'' k^6 \log^{C}p})  \right] +  \sum_{x< p' < p  \leq 2x}\left[  \exp(- \frac{n}{C' k^{6} \log^C p}) + \frac{Ck^2}{p'}  \right] 
\Big\}  \\ &\qquad\qquad  + \frac{2C_1 \log x}{x} \sum_{x < p \leq 2x} \left[  \exp\left(-\frac{n} {C' k^6 \log^{C}p}\right)+ \frac{Ck^2}{p} \right] \\
&\qquad\leq  C_3 \left[ \frac{ \log x}{x} \,  k^k +\exp\left(-\frac{n} {C'' k^6 \log^{C}x}\right)\right]~,
\end{split}\]
which implies, in conjunction with the Chebyshev inequality, the claimed bound.
\end{proof}

\section{Improved bound on the mixing}\label{s:upgr}

\begin{proof}[Proof of Lemma~\ref{cor:mix-upgr}]
According to a result of Breuillard--Gamburd \cite{BG}, there exists a set of primes $\mathcal P$ such that for large $x$
\[ \# (\text{primes} \setminus \mathcal P) \cap (x, 2x] \leq \sqrt x~,\]
and, for any $p \in \mathcal P$ and any generating (multi-)set $g_1, \cdots, g_m$ of $\PSL_2(p)$, the Markov chain with transfer probabilities $\mathbb P (g \to g_j^{\pm 1}g) = 1/(2m)$ has spectral gap bounded below by $\eta(m) > 0$, depending on $m$ but not on $p$ and the generators $g_j$. Invoking the result of Golsefidy--Srinivas \cite{GS}, we obtain that the spectral gap of the corresponding Markov chain on $\PSL_{2}(p_1) \times \cdots \times \PSL_{2}(p_k)$,
$p_j \in \mathcal P \cap (x, 2x]$, is lower-bounded by $\eta'(m,k)$, depending on $m$ and $k$. We apply this bound with $m=64$ and $g_j$  being one of the two largest atoms of the distribution of $V_1$.

Using the improved bound on the spectral gap in place of (\ref{eq:dsc}), we obtain the following improved version of Proposition~\ref{prop:mix}:  for any $k$, any prime $p_1 \geq \cdots \geq p_k \geq 5$ lying in $\mathcal P$, any $n \geq C_k \log p_1$, and any $\lambda_1 \in \mathbb F_{p_1}, \ldots, \lambda_k \in \mathbb F_{p_k}$ such that $\lambda_i \neq \lambda_{j}$ whenever $p_i = p_j$ and $i \neq j$, we have 
\begin{equation}\label{eq:impr-prob}
\begin{split} 
  \sum_{(g_1, \ldots, g_k) \in \PSL_2(p_1)\times \cdots \times \PSL_2(p_k)} \left| \mathbb P \left\{ \forall 1 \leq j \leq k \, : \, \Phi_n(\lambda_j) = g_j \right\} - \prod_{j=1}^k N_{p_j}^{-1} \right| 
    \leq \exp(-c_k n)~.
\end{split}
\end{equation}
Repeating the arguments in the proof of Corollary~\ref{cor:mix}, we deduce:
\begin{equation}  \mathbb P \left\{ \left| \frac{1}{x}  \sum_{p \in  \mathcal P \cap (x, 2x]} \log p \times  \left(  \# ( Z(P_n, p))_k  - 1  \right) \right| \geq \frac{1}{3} \right\}  
\leq  C_k \left[ \exp(-c_k n)+\frac{\log x}{x} \right] ~. \end{equation}
On the other hand, $P_n$ is a monic polynomial of degree $n$, hence $| \# (Z(P_n, p))_k - 1| \leq (n)_k \leq n^k$, and
\begin{equation}    \frac{1}{x}  \sum_{p \in (x, 2x] \setminus \mathcal P} \log p \times \left| \# ( Z(P_n, p))_k  - 1 \right| \leq \frac{C n^k}{\sqrt x} \leq \frac16 \end{equation}
provided that $x \geq C^2 n^{2k}$.
This concludes the proof.
\end{proof}

\section{Ruling out multiplicity}\label{s:mult}

The following argument mimics \cite[Section 8]{BV}.
\begin{proof}[Proof of Lemma~\ref{l:mult}] 
It is sufficient to estimate the probability that $P_n$ is an $m$-th power separately for each $2 \leq m \leq \sqrt{n}$; then we use the union bound. If $P_n = Q^m$, each value $P_n(\lambda)$, $\lambda \in \mathbb Z$, is an $m$-th power modulo any $p$. Choose a prime $p$  so that $m | p- 1$ and let $A \subset \mathbb F_{p}$ be the set of residues which are $m$-th powers modulo  $p$. Then $\# A =\frac{p-1}{m}+1 \leq \frac{p+1}{2}$, thus by Proposition~\ref{prop:mix}
\[ \mathbb{P} \left\{ \text{$P_n$ is an $m$-th power} \right\} 
\leq \mathbb P \left\{ P_n(\lambda) \in A~, \,\, 1\leq \lambda\leq k\right\} 
\leq \left( \frac{p+1}{2p} \right)^k + \exp(- \frac{n}{C k^6 \log^C p})\]
for $n \geq C k^7 \log^C p$. If $k = \lfloor n^c \rfloor$, where $c>0$ is a small positive number, we can ensure that $x < p \leq 2x$ for $x = \exp((n/k^7)^{1/(2C)})$, and then both terms in the right-hand side are bounded by $\exp(-n^c)$ for large $n$ (uniformly in $m$).
\end{proof}

\noindent Here is another version of the same argument, using slightly more advanced tools and yielding a better (exponential rather than stretched-exponential)  bound on the probablity.
\begin{proof}[Proof of Lemma~\ref{l:mult-upgr}] 
If $P_n$ is an $m$-th power, $P_n(0)$ is an $m$-th power modulo any prime $p$. It is known 
that the matrices $g_1 g_2 g_3^{-1} g_4^{-1}$, where $g_j \in \{ T(v), T(v')\}$, generate a Zariski-dense subgroup of $\SL_2(\mathbb C)$; this also follows from the fact (Lemma~\ref{l:gen} above) that these matrices generate $\SL_2(p)$ for $p$ large enough (in fact, by a result of Lubotzky \cite{Lub} it suffices to assume this property for one $p \geq 5$).

Hence, by the super-strong approximation theorem for $\SL_2$ of Bourgain--Gamburd--Sarnak \cite[Theorem 1.2]{BGS} there exists $x_0\in \mathbb Z$ and $\eta > 0$ such that for square-free $q$ coprime to $x_0$ the reduction of the corresponding random walk on $\SL_2(\mathbb Z / q 
\mathbb Z)$ has a spectral gap lower-bounded by $\eta$. Let $\frac{1}{10 x_0!} e^{\frac{\eta}{100}n} \leq q \leq e^{\frac{\eta}{100}n}$ be a  square-free integer not divisible by primes $\leq x_0$, then (for $c'$ possibly depending on $\eta$)
\[\begin{split} \mathbb{P} \left\{ \text{$P_n$ is an $m$-th power} \right\} 
&\leq \mathbb P \left\{ \text{$P_n(0)$ is a square $\operatorname{mod} q$}  \right\}\\
&\leq   \prod_{p | q} \frac{p+2}{2(p+1)} + \exp(- c \eta n) \leq \exp(-c'n)~. \qedhere
\end{split}\]

\end{proof}

\section{Matrices with constant diagonal entries}\label{s:dyson}

\begin{proof}[Proof of Theorem~\ref{thm:2}] We may assume without loss of generality that $a=0$: for an arbitrary $a$, 
$\det (\tilde H_n - \lambda \mathbbm 1) = \det ((\tilde H_n - a \mathbbm 1) - (\lambda - a) \mathbbm 1)$, where the right-hand side    is the characteristic polynomial of a shifted matrix with  $V_1\equiv 0$ evaluated at $\lambda-a$.
	
 Then the characteristic polynomial $P_n$ of  $\tilde H_n$ obeys the recurrence
\begin{equation}\label{eq:d-rec}  P_n(\lambda) = \lambda   P_{n-1}(\lambda) - W_{n-1}^2   P_{n-2}(\lambda)~.\end{equation}
In particular,  $P_n(0) = 0$ for odd $n$ (as we already showed in Claim~\ref{cl:d-1}), whereas for even $n$,  $P_n(0) = (-1)^{n/2} \prod_{j=1}^{n/2} W_{2j-1}^2$, whence deterministically $P_n(0) \neq 0 \mod p$ for  $p$ sufficiently large ($p \geq p_0$, where $p_0$ does not depend on $n$).

To count the zeros $\lambda \neq 0$, we first introduce some notation. For a positive integer $k$, let 
\begin{equation}
\label{eq:number_orbits}
\mathcal O(k) = \sum_{\substack{k_1,k_2\ge 0 \\ k_1+2k_2=k} }\frac{k!}{k_1! 2^{k_2}k_2!}.
\end{equation}
Note that $\mathcal O(k)$ is the number of partitions of a set of size $k$ into some subsets of size $2$ and one additional subset; equivalently, it is the number of $\sigma \in S_k$ with $\sigma^2=1$. Also, if $m$ is a positive integer with  $m\ge k$, then, $\mathcal O(k)$ is the number of orbits of the natural (imprimitive) action of $C_2\wr S_m$ on $(\Omega)_k$, where $\Omega$ is a set of size $2m$.

We prove the following counterpart of Lemma~\ref{cor:mix-upgr}:

\begin{lemma}\label{l:dyson-mix-upgr} For any $k \geq 1$ there exist $C_k, c_k>0$ such that for $c_k x^{\frac{1}{2k}} \geq n \geq C_k \log x$
\begin{equation}\label{eq:cor:mix-upgr} \mathbb P \left\{ \left| \frac{1}{x}  \sum_{x < p \leq 2x} \log p \times \# ( Z(P_n, p) \setminus \{0\})_k  - \mathcal O(k)  \right| \geq \frac{1}{2} \right\}  
\leq  C_k \left[ e^{-c_kn}+\frac{\log x}{x} \right] ~. \end{equation}
\end{lemma}

\begin{proof}
Let $\lambda_1, \cdots, \lambda_k \in \mathbb F_p \smallsetminus \{0\}$  be such that $\lambda_i \neq \pm \lambda_j$, $i \neq j$. Consider the transfer matrices, which now take the form
\[ \Phi_n(\lambda) = T(\lambda, |W_{n-1}|) \cdots T(\lambda, |W_{1}|) T(\lambda, |W_{0}|)~, \quad T(\lambda,w) = \left( \begin{array}{cc} \lambda & -w^2 \\ 1 & 0 \end{array}\right) \]
(and we can set $W_0 = 1$). Similarly to the proof of Lemma \ref{lem:computecharpolyfromtransfer}, we have $P_n(\lambda)=(\Phi_n(\lambda))_{11}$.
Note that these matrices are not unimodular, hence we consider 
\begin{equation}\label{eq:normalise1} \tilde T(\lambda, w) = \left( \begin{array}{cc} \lambda/w & -w \\ 1/w & 0 \end{array}\right) \in \PSL_2(p)~, \quad
\tilde\Phi_n(\lambda) = \tilde T(\lambda, |W_{n-1}|) \cdots \tilde T(\lambda, |W_{1}|) \tilde T(\lambda, |W_{0}|)~.\end{equation}

Now we prove two algebraic claims. The first one is a counterpart of Lemma~\ref{l:gen}.
\begin{cl}\label{cl:alg-dyson} Let $w,w' \neq 0$, $w\neq \pm w'$. Then for    $p \geq C(w,w')$ and $\lambda \in \mathbb F_p \setminus \{0\}$ we have $\langle S^3S^{-3}\rangle = \langle S^{-3}S^{3}\rangle  = \PSL_2(p)$, where $S=\{\tilde T(\lambda, w),   \tilde T(\lambda, w')\}$. 
\end{cl}

\begin{proof} 
First, we record part of one cycle of the projective action of $\tilde T(\lambda, w)$:
\begin{equation}
\label{eq:cycle}
\tilde T(\lambda, w) = (\ldots, w^2/\lambda, 0, \infty, \lambda, \lambda - w^2/\lambda,\ldots)  
\end{equation}
Next, note that 
\[
 \tilde T(\lambda, w)  \tilde T(\lambda, w')^{-1} =  \left( \begin{array}{cc}  w/w' & \lambda (w'/w - w/w') \\ 0 & w'/w \end{array}\right)\in S^3S^{-3}
\]
is of order equal to the order of $w^2/(w')^2\mod p$. Since $p$ is large, we may ensure that this order is at least $6$, and in particular the maximal subgroups (2), (4), (5), and (6) are ruled out (we use the labelling of the maximal subgroups as in the proof of Lemma~\ref{l:gen}). Note that $\tilde T(\lambda, w)  \tilde T(\lambda, w')^{-1}$ fixes only the points $\infty = [1:0]$ and $\lambda = [\lambda:1]$.  For $w_i\in \{w,w'\}$, by \eqref{eq:cycle} we have
\begin{equation}\label{eq:appltolam}  \begin{split}
&S^3S^{-3} \ni \tilde T(\lambda, w_1) \tilde T(\lambda, w_2) \tilde T(\lambda, w) \tilde T(\lambda, w')^{-1} \tilde T(\lambda, w')^{-1} \tilde T(\lambda, w')^{-1} ( \lambda) \\
&\qquad=  \tilde T(\lambda, w_1) \tilde T(\lambda, w_2) \tilde T(\lambda, w) \big( (w')^2/\lambda \big).
\end{split}
\end{equation}
If $\tilde T(\lambda, w)((w')^2/\lambda)\neq w^2/\lambda$, then we may choose $w_1=w_2=w$ and the above point is neither $\lambda$ nor $\infty$. If $\tilde T(\lambda, w)((w')^2/\lambda) = w^2/\lambda$, then we may choose  $w_1=w_2=w'$ and the same conclusion holds.
Similarly, 
\begin{equation} \label{eq:appltoinf}\begin{split}
&S^3S^{-3} \ni \tilde T(\lambda, w) \tilde T(\lambda, w) \tilde T(\lambda, w) \tilde T(\lambda, w)^{-1} \tilde T(\lambda, w')^{-1} \tilde T(\lambda, w')^{-1} (\infty) \\
&\qquad=  \tilde T(\lambda, w) \tilde T(\lambda, w) \big( (w')^2/\lambda \big)
\end{split}\end{equation}
is neither $\lambda$ nor $\infty$. Since $\tilde T(\lambda, w)  \tilde T(\lambda, w')^{-1}$ fixes only $\infty = [1:0]$ and $\lambda = [\lambda:1]$, and moreover has all nontrivial orbits of size at least $6$, it follows that $S^3S^{-3}$ does not stabilize neither a point nor a pair of points, so the maximal subgroups (1) and (3) are ruled out. Therefore, $S^3S^{-3}$ generates $\PSL_2(p)$ as claimed, and hence the same is true for $S^{-3}S^{3}$.
\end{proof}

Now we prove the counterpart of Lemma \ref{l:gen_prod}.
\begin{cl}\label{cl:dyson-2}  Let $w,w' \neq 0$, $w\neq \pm w'$. Then for  any  $k\geq 1$, $p_1, \cdots, p_k \geq C(w,w')$, and $\lambda_1\in  \mathbb F_{p_1} \setminus \{0\}, \cdots, \lambda_k\in  \mathbb F_{p_k} \setminus \{0\}$ such that $\lambda_j \neq \pm \lambda_i$ whenever $p_j = p_i$, we have 
\[ \langle S^3S^{-3}\rangle = \langle S^{-3}S^{3}\rangle  = \PSL_2(p_1) \times  \cdots \times \PSL_2(p_k)~,\] 
where 
\[ S=\{(\tilde T(\lambda_1,w),\tilde T(\lambda_2,w), \cdots, \tilde T(\lambda_k,w)), (\tilde T(\lambda_1,w'),\tilde T( \lambda_2,w'), , \cdots, \tilde T(\lambda_k,w'))\}~.\] 
\end{cl}

\begin{proof}
As in the proof of Lemma~\ref{l:gen_prod}, using Goursat's lemma we only need to consider the case $k=2$, $p_1 = p_2$.  Let $\lambda_2 \neq \pm \lambda_1$, and assume by contradiction there exists $h\in \PGL_2(p)$ such that 
\[\begin{split}
&\tilde T(\lambda_2, a_1) \tilde T(\lambda_2, a_2) \tilde T(\lambda_2, a_3)  \tilde T(\lambda_2, a_4)^{-1} \tilde T(\lambda_2, a_5)^{-1} \tilde T(\lambda_2, a_6)^{-1}   \\
&\quad= h \tilde T(\lambda_1, a_1) \tilde T(\lambda_1, a_2) \tilde T(\lambda_1, a_3)  \tilde T(\lambda_1, a_4)^{-1} \tilde T(\lambda_1, a_5)^{-1}\tilde T(\lambda_1, a_6)^{-1}  h^{-1} \end{split}\]
for every $a = (a_1, \cdots, a_6) \in \{w, w'\}^6$. (This is an equality in $\PSL_2(p)$.) Since $g_i=\tilde T(\lambda_i,w)\tilde T(\lambda_i,w')^{-1}$ fixes only $\lambda_i$ and $\infty$, the equality $g_2 = h g_1 h^{-1}$ implies that $h$ maps $\{\lambda_1, \infty\}$ to $\{\lambda_2, \infty\}$. Replacing $g_i$ by $\tilde T(\lambda_i,w)g_i\tilde T(\lambda_i,w)^{-1}$, by \eqref{eq:cycle} we see that $h$ maps  $\{\lambda_1, \lambda_1 - w^2/\lambda_1\}$ to $\{\lambda_2, \lambda_2 - w^2/\lambda_2\}$. This implies that $h$ fixes $\infty$, maps $\lambda_1$ to $\lambda_2$, and $\lambda_1 - w^2/\lambda_1$ to $\lambda_2-w^2/\lambda_2$, 
so
\[
h=\begin{pmatrix}
    a & (\lambda_2 - a^2\lambda_1)/a \\
    0 & a^{-1}
\end{pmatrix}
\]
where $a^2 = \lambda_1/\lambda_2$. The same argument applied to $\tilde T(\lambda_i,w)^2g_i\tilde T(\lambda_i,w)^{-2}$ implies (with \eqref{eq:cycle}) that $h\tilde T(\lambda_1, w)h^{-1}$ acts as $\tilde T(\lambda_2, w)$ on at least three distinct points. Since only the identity in $\PSL_2(p)$ fixes three points, we deduce that $h\tilde T(\lambda_1, w)h^{-1} = \tilde T(\lambda_2, w)$. In particular, $h$ fixes $0$, and so $\lambda_2 = a^2\lambda_1$, which implies $\lambda_1^2=\lambda_2^2$, against the assumption.
\end{proof}

Having Claim~\ref{cl:dyson-2} at hand, we adapt the arguments presented for the model (\ref{eq:defHn}) (proof of Lemma~\ref{cor:mix-upgr}), as follows. First, the arguments parallel to those in the proof of Lemma~\ref{cor:mix-upgr} (Section~\ref{s:upgr}) yield that $\tilde \Phi_n$ satisfy (\ref{eq:impr-prob})  for any $p_1 \geq \cdots \geq p_k \geq 5$ lying in $\mathcal P$, any $n \geq C_k \log p_1$, and any $\lambda_1 \in \mathbb F_{p_1}^\times, \cdots, \lambda_k \in \mathbb F_{p_k}^\times$ such that $\lambda_i \neq \pm \lambda_j$ whenever $p_i = p_j$ and $i \neq j$.

For a prime $p$, denote by $A_{p, k} \subset (\mathbb F_p^\times)_k$ the set of $k$-tuples such that $\lambda_i \neq - \lambda_j$ for any $i, j$. Also denote by $c(\sigma)$ the number of cycles of a permutation $\sigma$. Then we have
\begin{align}
\label{eq:estimate_mean}
\mathbb E\#(Z(P_n,p))_k  &= \sum_{(\lambda_1, \ldots, \lambda_k)\in (\mathbb F_p^\times)_k} \mathbb P(P_n(\lambda_i)=0 \text{ for every } i) \nonumber \\
&= \sum_{\substack{\sigma\in S_k \\ \sigma^2=1}} \sum_{(\mu_1, \ldots, \mu_{c(\sigma)})\in A_{p,c(\sigma)}} \mathbb P(P_n(\mu_i)=0 \text{ for every } i).
\end{align}
Next, for every $j$ 
\begin{align}
\label{eq:probability_one}
\sum_{(\mu_1, \ldots, \mu_j)\in A_{p,j}} \mathbb P(P_n(\mu_i)=0 \text{ for every } i) = \sum_{(\mu_1, \ldots, \mu_j)\in A_{p,j}} \sum_{\substack{g_i\in \PSL_2(p) \\ (g_i)_{11}=0}} \mathbb P(\tilde \Phi_n(\lambda_i)=g_i \text{ for every } i).
\end{align}
By the counterpart of (\ref{eq:impr-prob}) that we have just proved, for $p \in \mathcal P$ and $n \geq C \log p$ the difference between this quantity and $1$ is bounded in absolute value by $C_j ( e^{-c_jn}+\frac{1}{p})$, hence  the difference between  \eqref{eq:estimate_mean} and $\mathcal O(k)$ is bounded in absolute value by $C_k(e^{-c_kn}+\frac{1}{p})$. 
Similarly to Section~\ref{s:upgr}, this implies that
\[ \  \left| \mathbb E \,  \frac{1}{x}  \sum_{p \in \mathcal P \cap (x, 2x]} \log p \times \# ( Z(P_n, p) \setminus \{0\})_k  - \mathcal O(k)  \right|
\leq  C_k \left[ e^{-c_kn}+\frac{\log x}{x} \right]  \]
for $n \geq C_k \log x$. 
By a similar argument, we establish the same bound for the variance of
\[ \frac{1}{x}  \sum_{p \in \mathcal P \cap (x, 2x]} \log p \times \# ( Z(P_n, p) \setminus \{0\})_k~, \]
and then the proof of  Lemma \ref{l:dyson-mix-upgr} is concluded as in Section~\ref{s:upgr}.
\end{proof}

\noindent Bounding the height of $P_n$ similarly to Corollary~\ref{cor:height}  and invoking (\ref{eq:bv}),  we obtain 
\begin{equation}\label{eq:count-dyson} 
\mathbb P \Big\{ \# ( (Z(P_n)\smallsetminus \{0\})_k / \operatorname{Gal}(P_n / \mathbb Q) ) = \mathcal O(k) \Big\} \geq 1 - C \exp(-cn)~.\end{equation}
To deduce that $\operatorname{Gal}(P_n/\mathbb Q)$ contains the subgroup $\u[m] \rtimes A_m$
of $C_2 \wr S_m$, where $m=\lfloor n/2\rfloor$ and $\u[m]$ is as in (\ref{eq:def-Um}), 
we need the following group-theoretic lemma.

\begin{lemma}\label{l:dyson}
Let $\Omega$ be a set of cardinality $2m$, $m \geq 6$. Suppose $G\le C_2\wr S_m \le S_{2m}=\mathrm{Sym}(\Omega)$ preserves a decomposition $\Omega=\Omega_1\cup \cdots \cup \Omega_m$ into blocks of size $2$, and assume that the action of $G$ on the 6-tuples $(\Omega)_6$ has $\mathcal O(6)=76$ orbits. Then $G$ contains $\u[m]\rtimes A_m$.
\end{lemma}

We first compute some group cohomologies. In the following lemma we use additive notation, i.e., we write $G$-modules additively. The lemma is known (see \cite[Exercise 6.3]{aschbacher} for a similar statement) but we give a proof for completeness.

\begin{lemma}
\label{l:cohomology_symmetric}
	For $n\ge 3$, let $S_n$ act on $\mathbb F_2^n$ by permuting coordinates; we identify the subspace $(\mathbb F_2^n)^{S_n}$  of constant vectors with $\mathbb F_2$. Then:
	\begin{itemize}
		\item[(i)] $H^1(S_n, \mathbb F_2^n)\cong \mathbb F_2$ and $H^1(A_n,\mathbb F_2^n)=0$.
		\item[(ii)] $H^1(S_n,\mathbb F_2^n / \mathbb F_2)= H^1(A_n,\mathbb F_2^n / \mathbb F_2)=0$.
	\end{itemize}    
\end{lemma}

\begin{proof} 
	(i) Let $G=A_n$ or $S_n$ and let $H$ be the stabilizer of a point in $G$. We have $\mathbb F_2^n=\mathrm{Ind}_{H}^G(\mathbb F_2)$, where $\mathbb F_2$ is the trivial module, and so by Shapiro's lemma
	\[
	H^1(G,\mathbb F_2^n) = H^1(H,\mathbb F_2) = \mathrm{Hom}(H,\mathbb F_2),
	\]
	which is $\mathbb F_2$ for $G=S_n$, and $0$ for $G=A_n$.
	
	(ii) We first prove the statement for $G=A_n$. Assume first $n$ is odd. Then $\mathbb F_2^n = \mathbb F_2 \oplus \mathbb F_2^\perp$ where $\mathbb F_2^\perp$ is the subspace of $\mathbb F_2$ consisting of vectors whose sum is zero. Therefore $ \mathbb F_2^\perp \cong \mathbb F_2^n / \mathbb F_2$ and by (i) we get
	\[
	0= H^1(A_n, \mathbb F_2^n) \cong H^1(A_n,\mathbb F_2)\oplus H^1(A_n, \mathbb F_2^\perp) \cong  H^1(A_n,\mathbb F_2^n /\mathbb F_2)
	\]
	and so $H^1(A_n,\mathbb F_2^n /\mathbb F_2)=0$, as wanted.
	
	Now assume $n$ is even, and let $\mathbb F_2^\perp$ be as above (now $\mathbb F_2^\perp \supset \mathbb F_2$).  We first prove that 
\[ H^1(A_n, \mathbb F_2^\perp / \mathbb F_2)\cong \mathbb F_2~. \]
Let  $f\colon A_n \to \mathbb F_2^\perp / \mathbb F_2$ be a cocycle, so $f(\sigma\tau) = f(\sigma)+ \sigma f(\tau)$. Consider the stabiliser  $A_{n-1}$ of the point $n\in \{1,\ldots, n\}$. Note that the restriction of $\mathbb F_2^\perp / \mathbb F_2$ to $A_{n-1}$ is isomorphic to $\mathbb F_2^{n-1} / \mathbb F_2$, so by the statement with $n$ odd we have $H^1(A_{n-1},\mathbb F_2^\perp / \mathbb F_2)=0$, and we may assume  $f|_{A_{n-1}}=0$.
	
	Let $\sigma\in {A_{n}}$. Then, for every $\tau\in A_{n-1}$, we have $f(\sigma\tau) = f(\sigma)$, so $f$ is constant on each left coset of $A_{n-1}$ in $A_{n}$. Now let $\sigma\in A_{n}$ be such that $ \sigma^2\in {A_{n-1}}$.
We have $0=f(\sigma^2) = f(\sigma) + \sigma f(\sigma)$, so $f(\sigma) = \sigma f(\sigma)$. For  $\tau \in A_n \setminus A_{n-1}$, we consider all  $\sigma \in \tau A_{n-1}$ such that $\sigma^2 \in A_{n-1}$,  i.e.\ $\sigma (n) = \tau(n)$ and $\sigma(\tau (n)) =n$.  Then
\begin{equation}\label{eq:nzcocycle}f(\tau) = f(\sigma) = \sigma f(\sigma) = \sigma f(\tau)~.\end{equation}
Thus either $f(\tau) = 0$, or $f(\tau) = e_n + e_{\tau(n)} + \mathbb F_2$, where $e_i$ is the $i$-th coordinate vector. If $f$ is a cocycle which is not identically zero, the second possibility has to hold for all $\tau \in A_n$.
It is easy to see that (\ref{eq:nzcocycle}) indeed defines  a cocycle which is not a coboundary (since $n$ is even). Therefore, $H^1(A_n,\mathbb F_2^\perp/\mathbb F_2)\cong \mathbb F_2$, as claimed. 
	
	Now consider the short exact sequence $0\to \mathbb F_2^\perp / \mathbb F_2 \to \mathbb F_2^n / \mathbb F_2 \to \mathbb F_2^n / \mathbb F_2^\perp \to 0$. Part of the long exact sequence in cohomology is

\[
	\cdots \to  \stackbelow{0}{H^0(A_n,\mathbb F_2^n/ \mathbb F_2)} \to \stackbelow{\mathbb F_2}{H^0(A_n, \mathbb F_2^n / \mathbb F_2^\perp)} \to \stackbelow{\mathbb F_2}{H^1(A_n, \mathbb F_2^\perp / \mathbb F_2)} \to H^1(A_n,\mathbb F_2^n/ \mathbb F_2) \to \stackbelow{0}{H^1(A_n, \mathbb F_2^n / \mathbb F_2^\perp)} \to \cdots 
\]
	from which $H^1(A_n,\mathbb F_2^n/ \mathbb F_2) = 0$, as desired. 
	
	It remains to prove the statement for $G=S_n$. It is straightforward to do this by taking a cocycle $f\colon S_n \to \mathbb F_2^n/ \mathbb F_2$, restricting it to $A_n$, and using that $H^1(A_n,\mathbb F_2^n/ \mathbb F_2)=0$.
\end{proof}

\begin{proof}[Proof of Lemma \ref{l:dyson}] 
   By assumption, the number of orbits of $G$ on $(\Omega)_6$ is the same as the number of orbits of $C_2\wr S_m$. In particular, letting $\pi\colon G \to S_m$ be the action on the blocks, $\pi(G)$ is $6$-transitive, and so $A_m \le \pi(G)$. Moreover, if $\Delta$ is the union of $t\le 3$ blocks, then the stabilizer of $\Delta$ in $G$ induces $C_2\wr S_t$ on $\Delta$.

   Let $H=G\cap C_2^m$. We have that $H$ is normalized by $G/H\cong \pi(G)\ge A_m$; hence if $H$ contains a  vector, then it contains all of its even co\"ordinate permutations. This easily implies that  either $H\le (C_2^m)^{A_m}=\langle (-1, \ldots, -1)\rangle$ or $\u[m]\le H$.

   Assume first $\u[m]\le H$. Since $\u[m]$ is normal in $G$, $[\u[m],G] \leq \u[m]$. Moreover, using that $\pi(G) \geq A_m$ and looking at elements of the form $ugug^{-1} = u(gug^{-1})$, where $u\in \u[m]$ has exactly two co\"ordinates equal to $-1$, we see that $[\u[m],G] = \u[m]$. In addition, $[G,G] \le \u[m]\rtimes A_m$ and  $\pi([G,G])= A_m$, so $G\ge [G,G] = \u[m]\rtimes A_m$, which concludes the proof in this case.

Now we rule out the case  $H\le  \langle (-1, \ldots, -1)\rangle$. Assume for now that $\pi(G)=S_m$; the case $\pi(G)=A_m$ is essentially identical, as pointed out at the end of the proof. Let $K\cong S_m$ be the subgroup of $C_2\wr S_m$ permuting the blocks. We will apply Lemma \ref{l:cohomology_symmetric} with multiplicative notation (since in this proof the module $C_2^m$ is written multiplicatively).
   
Assume first $H=1$, so $G$ is a complement of $C_2^m$ in $C_2^m \rtimes S_m$. By Lemma \ref{l:cohomology_symmetric}(i), $H^1(S_m, C_2^m) \cong  C_2$, and in particular, up to conjugation we may assume that either $G=K$ or 
\[ G=\{((-1)^{\mathrm{sign}(\sigma)}, \ldots, (-1)^{\mathrm{sign}(\sigma)})\sigma \mid \sigma \in K\}~.\]
In both cases, the stabilizer of the union of two blocks does not induce  $C_2\wr S_2$, a contradiction.

Now we can turn to the case $H= \langle (-1, \ldots, -1)\rangle$, so  $G/H$ is a complement of $C_2^m/H$ in $C_2^m/H \rtimes S_m$.
By Lemma \ref{l:cohomology_symmetric}(ii), $H^1(S_m,C_2^m/H)=1$, and so up to conjugation we may assume $G=\langle H, K\rangle$. Then again the stabilizer of the union of two blocks does not induce $C_2\wr S_2$, contradiction.

   The case $\pi(G)=A_m$ is identical, using that $H^1(A_m,C_2^m)=H^1(A_m,C_2^m/\langle (-1, \ldots, -1)\rangle)=1$, as stated in Lemma \ref{l:cohomology_symmetric}. This concludes the proof of Lemma~\ref{l:dyson}.
\end{proof}

Now we can conclude the proof of Theorem~\ref{thm:2}. Set $Q_n(\lambda) = P_n(\lambda)$ if $n$ is even, and $Q_n(\lambda) = P_n(\lambda)/\lambda$ if $n$ is odd, and set $m=\lfloor n/2\rfloor$. As we saw (Claim~\ref{cl:d-1}), the Galois group $G$ of $Q_n$ is contained in  $C_2\wr S_m$.

Using the case $k = 1$ of (\ref{eq:count-dyson}) and arguing similarly to the proof of Lemma~\ref{l:mult-upgr} in Section~\ref{s:mult} to rule out proper powers of irreducibles, we show that $Q_n$ is irreducible with probability exponentially close to one. Now invoking the case $k=6$ of  (\ref{eq:count-dyson}) and   Lemma~\ref{l:dyson}, we conclude  that $G$ contains $\u[m]\rtimes A_m$ with probability exponentially close to one.
\end{proof}

\end{document}